\documentclass[11pt,a4paper]{article}
\usepackage{amsmath,amssymb,amsfonts,amsthm}
\usepackage[pdftex]{graphicx, color}

\newcommand{\E}{{\bf{E}}}
\newcommand{\PP}{{\bf{P}}}

\newtheorem{tm}{Theorem}

\newtheorem{lem}{Lemma}

\begin{document}

\bibliographystyle{plain}
\parindent=0pt
\centerline{\LARGE \bfseries Degree distribution of an inhomogeneous}
\centerline{\LARGE \bfseries  random intersection graph}

\par\vskip 3.5em

\centerline{Mindaugas  Bloznelis, Julius Damarackas}

\vglue2truecm

\centerline{Vilnius University, Faculty of mathematics and informatics,} 
\centerline{Naugarduko 24, Vilnius,
03225, Lithuania}

%\footnotetext[*]{Faculty of Mathematics and Informatics, Vilnius University, 03225 Vilnius, Lithuania}
%\centerline{ Mindaugas Bloznelis\footnote{Research supported by the Research Council of Lithuania Grant MIP-053}}

\bigskip

% Faculty of Mathematics and Informatics, Vilnius University,

% Naugarduko 24, LT-03225 Vilnius, Lithuania

%E-mail:   mindaugas.bloznelis$@$mif.vu.lt

%http://www.mif.vu.lt/$\sim$bloznelis
%\par\vskip 3.5em

\begin{abstract}
We show the asymptotic degree distribution of the typical vertex of a sparse inhomogeneous random intersection graph. 

\par
\end{abstract}

\smallskip
{\bfseries key words:}  degree distribution, random graph, random intersection graph, power law
\par\vskip 2.5em


\par\vskip 2.5em

\section{Introduction}

Let $X_1,\dots, X_m, Y_1,\dots, Y_n$ be independent non-negative random variables such that
each $X_i$ has the  probability distribution $P_1$ and each $Y_j$ has the probability distribution
$P_2$. Given realized values $X=\{X_i\}_{i=1}^m$ and $Y=\{Y_j\}_{j=1}^n$
we define the random bipartite graph $H_{X,Y}$ with the bipartition $V=\{v_1,\dots, v_n\}$, $W=\{w_1,\dots, w_m\}$, 
where  edges $\{w_i,v_j\}$ are inserted 
with probabilities
$p_{ij}=\min\{1, X_iY_j(nm)^{-1/2}\}$ independently 
for each $\{i,j\}\in [m]\times [n]$.
The 
 {\it inhomogeneous} random intersection graph  $G(P_1,P_2, n,m)$ defines the adjacency relation on the  vertex 
set $V$: vertices $v',v''\in V$ are declared adjacent (denoted $v'\sim v''$) whenever
$v'$ and $v''$ have a common neighbour in $H_{X,Y}$. 

 The degree distribution of the typical vertex of the random  graph  $G(P_1,P_2, n,m)$ has been first considered  by
Shang \cite{Shang2010}.
The proof of the main result of  \cite{Shang2010} contains  gaps and the  result is incorrect in the regime where
 $m/n\to\beta\in (0,+\infty)$ as $m,n\to+\infty$. 
 We remark that this regime is of particular importance, because it leads to inhomogeneous graphs
 with the clustering property: the clustering coefficient $\PP(v_1\sim v_2|v_1\sim v_3, v_2\sim v_3)$ is bounded away from zero
 provided that
  $\E X_1^3<\infty$ and $\E Y_1^2<\infty$, see \cite{BloznelisKurauskas2012}. 
The aim of the present paper is to show the asymptotic degree distribution in the case where $m/n\to\beta$ for 
some $\beta\in (0,+\infty)$.

 We consider a sequence of graphs $\{G_n=G(P_1,P_2, n,m)\}$, where $m=m_n\to +\infty$ as $n\to \infty$, and
where $P_1, P_2$ do not depend on $n$.
We denote $a_i=\E X_1^i$, $b_i=\E Y_1^i$.
 By $d(v_j)=d_{G_n}(v_j)$ we denote the degree of a vertex $v_j$ in $G_n$ (the number of vertices adjacent to $v_j$ in $G_n$). We remark that for every $n$ the random variables
 $d_{G_n}(v_1),\dots, d_{G_n}(v_n)$ are identically distributed. In Theorem \ref{T1} below we show the asymptotic distribution of $d(v_1)$.

\begin{tm}\label{T1} Let $m,n\to\infty$. 

(i) Assume that $m/n\to 0$. Suppose that  that $\E X_1<\infty$. Then $\PP(d(v_1)=0)=1-o(1)$.

(ii) Assume that $m/n\to \beta$ for some $\beta\in (0,+\infty)$.
 Suppose  that $\E X_1^2<\infty$ and $\E Y_1<\infty$. Then  $d(v_1)$ converges in distribution to the random variable
\begin{equation}\label{d*1}
d_*=\sum_{j=1}^{\Lambda_1}\tau_j, 
\end{equation}
where $\tau_1,\tau_2,\dots$ are independent 
and identically distributed random variables independent of the random variable $\Lambda_1$.
They are distributed as follows. For $r=0,1,2,\dots$, we have
\begin{equation}\label{d*1++}
\PP(\tau_1=r)=\frac{r+1}{\E\Lambda_2}\PP(\Lambda_2=r+1)
\qquad
{\text{and}}
\qquad
\PP(\Lambda_i=r)=\E \,e^{-\lambda_i}\frac{\lambda_i^r}{r!},
\qquad
i=1,2.
\end{equation}
Here $\lambda_1=Y_1a_1\beta^{1/2}$ and $\lambda_2=X_1b_1\beta^{-1/2}$.

(iii) Assume that $m/n\to+\infty$.
 Suppose that  $\E X_1^2<\infty$ and $\E Y_1<\infty$. 
Then  $d(v_1)$ converges in distribution to a random variable
$\Lambda_3$ having the probability distribution
\begin{equation}\label{d*2}
\PP(\Lambda_3=r)=\E e^{-\lambda_3}\frac{\lambda_3^r}{r!},
\qquad
r=0,1,\dots.
\end{equation}
Here $\lambda_3=Y_1a_2b_1$.
\end{tm}

 {\it Remark 1.} The probability distributions $P_{\Lambda_i}$ of $\Lambda_i$, $i=1,2,3$, are Poisson mixtures. One way to
sample 
from the  distribution $P_{\Lambda_i}$  is to 
generate random variable  $\lambda_i$ and then, given $\lambda_i$, 
to generate Poisson random variable 
with the parameter $\lambda_i$. The realized value of the Poisson random variable 
has the  distribution $P_{\Lambda_i}$. 

{\it Remark 2}. The asymptotic distributions (\ref{d*1}) and (\ref{d*2}) admit  heavy tails.
In the case (ii) we obtain a power law asymptotic degree distribution (\ref{d*1}) provided that at least one 
of $P_1$ and $P_2$ has a power law and $P_1(0),P_2(0)<1$.
In the case (iii) 
we obtain a power law asymptotic degree distribution (\ref{d*2}) provided that  $P_2$ has a power law.

{\it Remark 3}. Since the second moment $a_2$ does not show up in (\ref{d*1}), (\ref{d*1++}) we expect that in the case 
(ii) the second moment condition $\E X_1^2<\infty$ is redundant and  could perhaps be replaced by the weaker first moment condition $\E X_1<\infty$.

 Random intersection graphs 
 have 
 attracted considerable attention in the recent literature, see, e.g., 
 \cite{Barbour2011},
 \cite{behrisch2007},
 \cite{Blackburn2009},
 \cite{Bradonjic2010},
 \cite{Britton2008},
 \cite{eschenauer2002},
 \cite{Rybarczyk2011}. 
 Starting with the paper  
  by Karo\'nski et al \cite{karonski1999}, see also \cite{Singer1995},
  where the  case  of degenerate distributions $P_1=P_{1n}$, $P_2=P_{2n}$ depending on $n$ was considered 
  (i.e., $\PP_{1n}(c_n)=\PP_{2n}(c_n)=1$,  for some $c_n>0$), 
 several more complex random intersection graph models were later introduced by Godehardt and Jaworski \cite{godehardt2001}, Spirakis et al. \cite{Spirakis2004}, Shang \cite{Shang2010}.  
%Guillaume and  Latapy
%\cite{Guillaume+L2004} argue that many complex networks expoit an underlying
%bipartite graph structure, like the bipartite graph $H$ above.
%Independently, Eshenauer and Gligor \cite{eschenauer2002} came up with a  random intersection graph  as a model of secure wireless %sensor network that uses random predistribution of keys. 
 The asymptotic degree distribution  for various random intersection graph models was shown in  
  \cite{Bloznelis2008}, \cite{Bloznelis2010}, \cite{Bloznelis2010+}, \cite{Bloznelis2011+},  \cite{Deijfen}, \cite{JKS}, \cite{JaworskiStark2008},
  \cite{Rybarczyk-degree2011}, \cite{stark2004}.

\section{Proofs}

Before the proof we introduce some notation and give two auxiliary lemmas.

The event that the edge $\{w_i, v_j\}$ is present in $H=H_{X,Y}$ is denoted $w_i\to v_j$. We denote
\begin{displaymath}
 {\mathbb I}_{ij}={\mathbb I}_{\{w_i\to v_j\}},
\qquad
{\mathbb I}_i={\mathbb I}_{i1},
\qquad
u_i=\sum_{2\le j\le n}{\mathbb I}_{ij},
\qquad
L=\sum_{i=1}^mu_i{\mathbb I}_i.
\end{displaymath}
We remark, that $u_i$ counts all neighbours of $w_i$ in $H$ 
belonging to the set $V\setminus \{v_1\}$. Denote
\begin{eqnarray}\nonumber
 &&
{\hat a}_k=m^{-1}\sum_{1\le i\le m}X^k_i,
\qquad
{\hat b}_k=n^{-1}\sum_{2\le j\le n}Y^k_j,
\\
\label{S-XY}
&&
\lambda_{ij}=\frac{X_iY_j}{\sqrt{mn}},
\qquad
 S_{XY}=\sum_{i=1}^m\sum_{j=2}^n\frac{X_i}{\sqrt{nm}}\lambda_{ij}\min\{1,\lambda_{ij}\},
\end{eqnarray}
and introduce the event ${\cal A}_1=\{\lambda_{i1}<1, \, 1\le i\le m\}$.
By ${\PP}_1$ and $\E_1$ we denote the conditional probability and conditional expectation given $Y_1$.
By ${\tilde \PP}$ and ${\tilde \E}$ we denote the conditional probability and conditional expectation 
given $X,Y$.
By  $d_{TV}(\zeta,\xi)$ we denote the total variantion 
distance between the probability distributions
of
 random variables $\zeta$ and $\xi$. In the case where $\zeta,\xi$ and $X, Y$ are defined on the same probability space we denote 
 by 
 ${\tilde d}_{TV}(\zeta,\xi)$ the total 
variation distance between the conditional distributions of $\zeta$ and $\xi$ given $X,Y$.

In the proof  below we use the following 
simple fact about the convergence of a sequence of random variables $\{\varkappa_n\}$: 
\begin{equation}\label{fact}
 \varkappa_n=o_P(1),
\quad
\E \sup_n\varkappa_n<\infty
\quad
\Rightarrow
\quad
\E\varkappa_n=o(1).
\end{equation}
We remark that the condition $\E \sup_n\varkappa_n<\infty$ (which assumes implicitly that
all $\varkappa_n$ are defined on the same probability space) can be replaced by the more 
resctrictive condition that there exists a constant $c>0$ such that 
$|\varkappa_n|\le c$, for all $n$. In the latter case we do not need all $\varkappa_n$ to be defined on the same probability space.
In particular, given a sequence of bivariate random vectors $\{(\eta_n,\theta_n)\}$ such that, for  every $n$ and $m=m_n$ random variables $\eta_n$, $\theta_n$ and
 $\{X_i\}_{i=1}^m$, $\{Y_j\}_{j=1}^n$ are defined on the same probability space, we have 
\begin{displaymath}
{\tilde d}_{TV}(\eta_n,\theta_n)=o_P(1)
\
\Rightarrow
\
d_{TV}(\eta_n,\theta_n)\le \E{\tilde d}_{TV}(\eta_n,\theta_n) =o(1).
\end{displaymath}

\begin{lem}\label{d-L} Assume that $\E X_1^2<\infty$ and $\E Y_1<\infty$.
We have as $n,m\to+\infty$
\begin{eqnarray}\label{d-L21}
&&
\PP(d(v_1)\not=L)=o(1),
\\
\label{d-L22}
&&
\PP({\cal A}_1)=1-o(1),
\\
\label{d-L23}
&&
S_{XY}=o_P(1),
\qquad
\E S_{XY}=o(1).
\end{eqnarray}

\end{lem}
\begin{proof}[Proof of Lemma \ref{d-L}] Proof of (\ref{d-L21}).
We observe that the event ${\cal A}=\{d(v_1)\not=L\}$ occurs in the case where for some $2\le j\le n$
and some distinct $i_1,i_2\in [m]$ the event 
${\cal A}_{i_1,i_2,j}=
\{w_{i_1}\to v_1,\, w_{i_1}\to v_j,\, w_{i_2}\to v_1,\, w_{i_2}\to v_j\}$ occurs.
From the union bound and the inequality 
${\tilde \PP}({\cal A}_{i_1,i_2,j})\le m^{-2}n^{-2}X_{i_1}^2X_{i_2}^2Y_1^2Y_j^2$ we obtain
\begin{equation}\label{d-L1}
{\tilde \PP}({\cal A})
=
{\tilde \PP}
\left(\bigcup_{\{i_1, i_2\}\subset[m]}\bigcup_{2\le j\le n}{\cal A}_{i_1,i_2,j}\right)
\le n^{-1}{\hat b}_2Y_1^2 Q_X.
\end{equation}
Here $Q_X=m^{-2}\sum_{\{i_1,i_2\}\subset[m]}X_{i_1}^2X_{i_2}^2$. We note that
$Q_X$ and  $Y_1^2$ are stochastically bounded  and 
%We note that the product $Y_1^2Q_X$ is stochastically bounded and 
$n^{-1}{\hat b}_2=o_P(1)$ 
%converges to zero in probability
as $n\to+\infty$. Therefore, 
%the random variable 
${\tilde \PP}({\cal A})=o_P(1)$. 
%converges to zero in probability. 
Now (\ref{fact}) implies (\ref{d-L21}).
%This implies (\ref{d-L21}) since, for any $\varepsilon\in (0,1)$,
%\begin{displaymath}
%\PP({\cal A})=\E ({\tilde \PP}({\cal A}))
%\le 
%\varepsilon+\PP({\tilde \PP}({\cal A})>\varepsilon)
%=\varepsilon+o(1).
%\end{displaymath}

\medskip

Proof of (\ref{d-L22}).
Let ${\overline {\cal A}}_1$ denote the complement event to ${\cal A}_1$.
%Next we show that $\PP({\overline{\cal A}}_1)=o(1)$. 
%We obtain from 
By 
the union bound  and Markov's inequality 
%that
\begin{displaymath}
\PP_1({\overline {\cal A}}_1)\le \sum_{i\in [m]}\PP_1(\lambda_{i1}\ge 1)
\le 
(nm)^{-1}Y_1^2\sum_{i\in [m]}\E X_i^2=n^{-1}a_2Y_1^2.
\end{displaymath}
Hence we obtain $\PP_1({\overline {\cal A}}_1)=o_P(1)$.  Now (\ref{fact}) implies $\PP({\overline {\cal A}}_1)=\E \PP_1({\overline {\cal A}}_1)=o(1)$.
%(\ref{d-L22}).
%Note that this implies (\ref{d-L22}) since, 
%for any $\varepsilon\in (0,1)$, we have
%\begin{displaymath}
% \PP({\overline{\cal A}}_1)
%=
%\E\PP_1({\overline {\cal A}}_1)
%\le 
%\varepsilon+\PP(\PP_1({\overline {\cal A}}_1)>\varepsilon)
%=
%\varepsilon+o(1).
%\end{displaymath}

\medskip

Proof of (\ref{d-L23}). Since the first bound of (\ref{d-L23}) follows from the second one, we only prove the latter.
%Now, we show that $S_{XY}=o_P(1)$.
Denote ${\hat X}_1=\max\{X_1, 1\}$ and ${\hat Y}_1=\max\{Y_1,1\}$.
We observe that $\E X_1^2<\infty$, $\E Y_1<\infty$ implies
\begin{displaymath}
\lim_{t\to+\infty}\E {\hat X}_1^2{\mathbb I}_{\{{\hat X}_1>t\}}=0,
\qquad
\lim_{t\to+\infty}\E {\hat Y}_1{\mathbb I}_{\{{\hat Y}_1>t\}}=0. 
\end{displaymath}
 Hence one can find
a strictly increasing function $\varphi:[1,+\infty)\to [0,+\infty)$ 
with $\lim_{t\to+\infty}\varphi(t)=+\infty$
such that
\begin{equation}\label{hat1}
\E {\hat X}_1^2\varphi({\hat X}_1)<\infty,
\qquad  
\E {\hat Y}_1\varphi({\hat Y}_1)<\infty.
\end{equation} 
In addition, we can choose $\varphi$  satisfying 
\begin{equation}\label{phi}
\varphi(t)<t
\qquad
{\text{and}}
\qquad
\varphi(st)
\le \varphi(t)\varphi(s), 
\qquad
\forall s,t\ge 1. 
\end{equation}
For this purpose we take a sufficiently slowly growing concave function 
$\psi:[0,+\infty)\to [0,+\infty)$ with $\psi(0)=0$ and define
$\varphi(t)=e^{\psi(\ln (t))}$. We note that the second inequality of (\ref{phi})
follows from the concavity property of $\psi$. We remark, that (\ref{phi})
implies
\begin{equation}\label{phi1}
 t/(st)=s^{-1}\le 1/\varphi(s)\le\varphi(t)/\varphi(st),
\qquad
s,t\ge 1.
\end{equation}

Let ${\hat S}_{XY}$ be defined as in  (\ref{S-XY}) above, but with
$\lambda_{ij}$ replaced by ${\hat\lambda}_{ij}={\hat X}_i{\hat Y}_j/\sqrt{mn}$.
We note, that $S_{XY}\le {\hat S}_{XY}$. Furthermore, from the inequalities 
\begin{equation}\label{phi2}
 \min\{1,{\hat\lambda}_{ij}\}
\le
\min\Bigl\{1,\frac{\varphi({\hat X}_i{\hat Y}_j)}{\varphi(\sqrt{mn})}\Bigr\}
\le
\frac{\varphi({\hat X}_i{\hat Y}_j)}{\varphi(\sqrt{mn})}
\le 
\frac{\varphi({\hat X}_i)\varphi({\hat Y}_j)}{\varphi(\sqrt{mn})}
\end{equation}
we obtain 
$S_{XY}\le {\hat S}_{XY}\le S^*_{XY}/\varphi(\sqrt{nm})$, 
where 
\begin{displaymath}
S^*_{XY}
=
(mn)^{-1}
\left(\sum_{1\le i\le m}{\hat X}^2_i\varphi({\hat X}_i)\right)
\left(\sum_{2\le j\le n}{\hat Y}_j\varphi({\hat Y}_j)\right).
\end{displaymath}
We remark that (\ref{phi}) and (\ref{phi1}) and  imply the third and the first inequality of 
(\ref{phi2}), respectively. 

Finally, the bound $\E S_{XY}=o(1)$ follows from the inequality 
$S_{XY}\le S^*_{XY}/\varphi(\sqrt{nm})$
and the fact  that $\E S^*_{XY}$ remains bounded as $n,m\to+\infty$, 
see  (\ref{hat1}).
\end{proof}

%*****************************************************************************************************************************

In the proof of Theorem \ref{T1} we use the following inequality 
 refered to as LeCam's lemma, see e.g., \cite{Steele}.
\begin{lem}\label{LeCamLemma} Let $S={\mathbb I}_1+ {\mathbb I}_2+\dots+ {\mathbb I}_n$ be the sum of independent random indicators
with probabilities $\PP({\mathbb I}_i=1)=p_i$. Let $\Lambda$ be Poisson random variable with mean $p_1+\dots+p_n$. The total variation
distance between the distributions $P_S$ of $P_{\Lambda}$ of $S$ and $\Lambda$
\begin{equation}\label{LeCam}
\sup_{A\subset \{0,1,2\dots \}}|\PP(S\in A)-\PP(\Lambda\in A)|= \frac{1}{2} \sum_{k\ge 0}|\PP(S=k)-\PP(\Lambda=k)|
\le
\sum_{i}p_i^2.
\end{equation}
\end{lem}

%*****************************************************************************************************************************

\begin{proof}[Proof of Theorem \ref{T1}]

{\it The case (i).} We have $\PP(d(v_1)>\varepsilon)\le \PP(L>\varepsilon)$, for each $\varepsilon>0$. We prove that $\PP(L>\varepsilon)=o(1)$, for any $\varepsilon>0$.
In view of the identity $\PP(L>\varepsilon)=\E\PP_1(L>\varepsilon)$ and (\ref{fact}) it suffices to show that $\PP_1(L>\varepsilon)=o_P(1)$.
For this purpose we write, by the union bound and Markov's inequality,  
\begin{displaymath}
\PP_1(L>\varepsilon)
\le 
\sum_{1\le k\le m}\PP_1({\mathbb I}_k=1)
\le
\E_1\sum_{1\le k\le m}\lambda_{k1}
=
\sqrt{m/n}Y_1\E X_1
=
o_P(1).
\end{displaymath}

In cases (ii) and (iii) we apply (\ref{d-L21}). In view of 
(\ref{d-L21}) the random variables $d(v_1)$ and $L$ have the same asymptotic distribution (if any).
Hence, it suffices to
%We will 
show 
the convergence in distribution of $L$.

{\it The case (ii)}.
Here we prove that  $L$ converges in distribution to (\ref{d*1}). 
 We first approximate $L$ by 
the  random variable $L_3=\sum_{k=1}^m\eta_k\xi_{3k}$. 
Then we show that $L_3$ converges in distribution to (\ref{d*1}).
Here $\eta_1,\dots, \eta_m$, $\xi_{3 1},\dots, \xi_{3 m}$ are 
conditionally independent (given $X,Y$) 
Poisson random variables with ${\tilde \E}\eta_k=\lambda_{k1}$ and 
${\tilde \E}\xi_{3 k}=X_k(n/m)^{1/2}b_1$. We assume, in addition, 
that given $X,Y$, the 
sequences $\{{\mathbb I}_{k}\}_{k=1}^m$ and $\{\xi_{3k}\}_{k=1}^m$ 
are conditionally independent.
 
Given $X,Y$, we generate independent Poisson random variables
$\xi_{11},\dots, \xi_{1m}$, $\Delta_{11},\dots, \Delta_{1m}$, with the conditional mean values
\begin{displaymath}
{\tilde \E}\xi_{1k}= \sum_{2\le j\le n}p_{kj},
\qquad 
{\tilde \E}\Delta_{1k}=\sum_{2\le j\le n}(\lambda_{kj}-p_{kj}),
\qquad
1\le k\le m.
\end{displaymath}
We assume that, given $X,Y$, these Poisson random variables are conditionally independent of the sequence 
$\{\eta_k\}_{k=1}^m$. We suppose, in addition, that $\{\eta_k\}_{k=1}^m$ is
conditionally  independent (given $X, Y_1$)  of the set of edges of $H$ 
that are not incident to 
$v_1$. 
We define  $\xi_{2k}=\xi_{1k}+\Delta_{1k}$ and observe that $\xi_{2k}$ has conditional (given $X,Y$) Poisson distribution with
the conditional mean value ${\tilde \E}{\xi}_{2k}=\sum_{2\le j\le n}\lambda_{kj}$.
Introduce the random variables 
\begin{displaymath}
L_0=\sum_{1\le k\le m}\eta_ku_k,
\qquad
L_1=\sum_{1\le k\le m}\eta_k\xi_{1k},
\qquad
L_2=\sum_{1\le k\le m}\eta_k\xi_{2k}.
\end{displaymath}
In order to show that $L$ and $L_3$ 
have the same asymptotic probability distribution (if any) we 
prove that 
\begin{eqnarray}\label{DeltaT}
&&
d_{TV}(L, L_0)=o(1),
\qquad
d_{TV}(L_0,L_1)=o(1),
\\
\label{L-L}
&&
\E|L_1-L_2|=o(1),
\qquad
\qquad
{\tilde L}_2-{\tilde L}_3=o_P(1).
\end{eqnarray}
Here ${\tilde L}_2$ and ${\tilde L}_3$ are marginals of the random 
vector $({\tilde L}_2, {\tilde L}_3)$ constructed below which has the 
property that ${\tilde L}_2$ has the same distribution as $L_2$ and ${\tilde L}_3$ has 
the same distribution as $L_3$.

Let us prove the first bound of (\ref{DeltaT}).
In view of (\ref{fact}) it suffices to show that
${\tilde d}_{TV}(L_0,L)=o_P(1)$. In order to prove the latter bound we apply the inequality
\begin{equation}\label{TV1}
 {\tilde d}_{TV}(L_0, L){\mathbb I}_{{\cal A}_1}\le n^{-1}Y_1^2{\hat a}_2
\end{equation} 
shown below. 
We remark that (\ref{TV1}) implies 
\begin{displaymath}
 {\tilde d}_{TV}(L_0,L)
 \le 
 {\tilde d}_{TV}(L_0,L) {\mathbb I}_{{\cal A}_1}+{\mathbb I}_{{\overline{\cal A}}_1}
 \le
 n^{-1}Y_1^2{\hat a}_2+{\mathbb I}_{{\overline{\cal A}}_1}=o_P(1).
\end{displaymath}
Here $n^{-1}Y_1^2{\hat a}_2=o_P(1)$, because $Y_1^2{\hat a}_2$ is stochastically bounded. Furthermore,  the bound
${\mathbb I}_{{\overline{\cal A}}_1}=o_P(1)$
follows from (\ref{d-L22}).

It remains to prove (\ref{TV1}).
% from LeCam's inequality, see Lemma \ref{LeCamLemma} below. 
We denote  
$L'_k=\sum_{i=1}^k{\mathbb I}_iu_i+\sum_{i=k+1}^m\eta_iu_i$ and write, by the triangle inequality, 
%of the total variation distance,
\begin{displaymath}
 {\tilde d}_{TV}(L_0,L)\le \sum_{k=1}^m{\tilde d}_{TV}(L'_{k-1}, L'_k).
\end{displaymath}
Then we estimate 
${\tilde d}_{TV}(L'_{k-1}, L'_k)
\le 
{\tilde d}_{TV}(\eta_k,{\mathbb I}_k)\le (nm)^{-1}Y_1^2X_k^2$.
Here the first inequality follows from the properties of the total variation distance. 
The second inequality follows from Lemma \ref{LeCamLemma} and the fact 
that on the event ${\cal A}_1$ we have
$p_{k1}=\lambda_{k1}$.

\medskip

Let us prove the second bound of (\ref{DeltaT}).
In view of (\ref{fact}) it suffices to  show that
  ${\tilde d}_{TV}(L_0,L_1)=o_P(1)$. We denote  
$L^*_k=\sum_{i=1}^k\eta_iu_i+\sum_{i=k+1}^m\eta_i\xi_{1i}$ and write, 
by the triangle inequality, 
%of the total variation distance,
\begin{equation}\label{dTV4}
 {\tilde d}_{TV}(L_0,L_1)\le \sum_{k=1}^m{\tilde d}_{TV}(L^*_{k-1}, L^*_k).
\end{equation}
Here
\begin{equation}\label{dTV41}
{\tilde d}_{TV}(L^*_{k-1}, L^*_k)\le {\tilde d}_{TV}(\eta_ku_k,\eta_k\xi_{1k})
\le 
{\tilde \PP}(\eta_k\not=0){\tilde d}_{TV}(u_k,\xi_{1k}).
\end{equation}
Now, invoking the inequalities  
${\tilde \PP}(\eta_k\not=0)=1-e^{-\lambda_{k1}}\le \lambda_{k1}$
and ${\tilde d}_{TV}(u_k,\xi_{1k})\le \sum_{j=2}^np^2_{kj}$, we obtain from 
(\ref{dTV4}), (\ref{dTV41})  and (\ref{d-L23}) that
\begin{displaymath}
 {\tilde d}_{TV}(L_0,L_1)\le \sum_{k=1}^m\lambda_{k1}\sum_{j=2}^np^2_{kj} 
\le 
Y_1S_{XY}=o_P(1).
\end{displaymath}

\medskip

Let us prove the first  bound of (\ref{L-L}).
We observe that 
\begin{displaymath}
 |L_2-L_1|=L_2-L_1
=
\sum_{1\le k\le m}\eta_k\Delta_{1k}
\end{displaymath}
and 
\begin{displaymath}
{\tilde \E}\sum_{1\le k\le m}\eta_k\Delta_{1k}
=
\sum_{1\le k\le m}\lambda_{k1}\sum_{2\le j\le n}
(\lambda_{kj}-1){\mathbb I}_{\{\lambda_{kj}>1\}}
\le Y_1S_{XY}. 
\end{displaymath}
We obtain $\E|L_2-L_1|\le \E Y_1\E S_{XY}=o(1)$, see
 (\ref{d-L23}).

\medskip

Let us prove the second  bound of (\ref{L-L}).  Given $X,Y$, generate 
independent Poisson random variables
$\xi'_{31},\dots,\xi'_{3m}, \Delta_{21},\dots,\Delta_{2m}, \Delta_{31},\dots, \Delta_{3m}$ which are conditionally independent of the sequence $\{\eta_k\}_{k=1}^m$ and have  the conditional mean values
\begin{displaymath}
{\tilde \E}\xi'_{3k}=X_k(n/m)^{1/2}b,
\qquad
{\tilde \E}\Delta_{2k}=X_k(n/m)^{1/2}\delta_{2},
\qquad
{\tilde \E}\Delta_{3k}=X_k(n/m)^{1/2}\delta_{3}.
\end{displaymath}
Here
$b=\min\{{\hat b}_1, b_1\}$, $\delta_{2}= {\hat b}_1-b$, $\delta_{3}=b_1-b$. We note that $\delta_2,\delta_3\ge 0$ and
observe that the random vector
\begin{displaymath}
(L'_2,L'_3),
\qquad
L'_2=\sum_{1\le k\le m}\eta_k(\xi'_{3k}+\Delta_{2k}), 
\qquad
L'_3=\sum_{1\le k\le m}\eta_k(\xi'_{3k}+\Delta_{3k})
\end{displaymath}
has the marginal distributions of $(L_2,L_3)$. In addition, we have 
\begin{equation}\label{L2-L3}
 {\tilde \E}|L'_2-L'_3|
=
(\delta_2+\delta_3)Y_1{\hat a}_2=|{\hat b}_1-b_1|Y_1{\hat a}_2=o_P(1).
\end{equation}
In the last step we used the fact that $Y_1{\hat a}_2=O_P(1)$ and ${\hat b}_1-b_1=o_P(1)$, 
by the law of large numbers.  
Finally, we show that (\ref{L2-L3}) implies
the bound $|L'_2-L'_3|=o_P(1)$. Denoting, for short, $H=|L_2'-L_3'|$ and 
$h={\tilde \E}{\mathbb I}_{\{H\ge \varepsilon\}}$ we write, for $\varepsilon\in (0,1)$,
\begin{equation}\label{HHHH}
 \PP(H\ge\varepsilon)
=
\E h
=
\E h
({\mathbb I}_{\{{\tilde \E}H\ge\varepsilon^2\}}
+
{\mathbb I}_{\{{\tilde \E}H<\varepsilon^2\}}).
\end{equation}
Using the simple inequality $h\le 1$ and the inequality,
$h\le \varepsilon^{-1}{\tilde \E}H$, which follows from Markov's inequality, we obtain
\begin{eqnarray}\nonumber
 \E h
{\mathbb I}_{\{{\tilde \E}H\ge\varepsilon^2\}}
&
\le 
&
\E 
{\mathbb I}_{\{{\tilde \E}H\ge\varepsilon^2\}}=\PP({\tilde \E}H\ge\varepsilon^2)=o(1),
\\
\nonumber
\E h {\mathbb I}_{\{{\tilde \E}H<\varepsilon^2\}}
&
\le 
&
\E (\varepsilon^{-1}{\tilde \E}H){\mathbb I}_{\{{\tilde \E}H<\varepsilon^2\}}
< 
\varepsilon.
\end{eqnarray}
Invoking these inequalities in (\ref{HHHH}) we obtain $\PP(H\ge \varepsilon)<\varepsilon+o(1)$.
Hence $H=o_P(1)$. 

% Indeed, denote
% $H:=|{\tilde L}'-{\bar L}'|$, ${\tilde H}:={\tilde E}H$ and write for 
%any $\varepsilon\in (0,1)$ 
%\begin{displaymath}
% \PP(H>\varepsilon)\le\PP({\tilde H}>\varepsilon^2)
%+
%\PP(H>\varepsilon, {\tilde H}\le\varepsilon^2).
%\end{displaymath}
%Next we apply Markov's inequality to the second probability
%\begin{displaymath}
% \PP(H>\varepsilon, {\tilde H}\le\varepsilon^2)
%=
%\E{\mathbb I}_{\{ H>\varepsilon \}} {\mathbb I}_{\{ {\tilde H}\le \varepsilon^2\}}
%\le
%\varepsilon^{-1}\E H{\mathbb I}_{\{{\tilde H}\le \varepsilon^2\}}
%\end{displaymath}
%and then invoke the simple inequality
%\begin{displaymath}
% \E (H{\mathbb I}_{\{{\tilde H}\le \varepsilon^2\}})
%=\E({\tilde \E} (H{\mathbb I}_{\{{\tilde H}\le \varepsilon^2\}}))=
%\E ({\mathbb I}_{\{{\tilde H}\le \varepsilon^2\}} {\tilde H})
%\le \varepsilon^2.
%\end{displaymath}
%We obtain $\PP(H>\varepsilon)\le\PP({\tilde H}>\varepsilon^2)+\varepsilon=o(1)+\varepsilon$,
%thus proving the bound $H=o_P(1)$.

Now we prove that  $L_3$ converges in distribution to (\ref{d*1}).
 Introduce the random variable ${\bar L}=\sum_{1\le k\le m}\eta_k{\bar \xi}_k$, where,
given $X,Y$, the random variables ${\bar \xi}_1,\dots, {\bar \xi}_m$ 
are conditionally independent of $\{\eta_k\}_{k=1}^m$ and have the 
conditional mean values ${\tilde \E}{\bar \xi}_k=X_k\beta^{-1/2}b_1$. 
Proceeding as in the proof of the second bound of 
(\ref{L-L}) above, we construct a random vector 
$(L_3'', {\bar L}')$ 
with the same marginals as $(L_3, {\bar L})$ 
and such that 
\begin{equation}\label{LLL}
{\tilde \E}|L''_3- {\bar L}'|
\le 
|1-(m/n)^{1/2}\beta^{-1/2}|Y_1b_1{\hat a}_2=o_P(1).
\end{equation}
In the last step we used the fact that $Y_1{\hat a}_2=O_P(1)$ and $m/n\to\beta$. 
Now, (\ref{LLL}) implies $L''_3-{\bar L}'=o_P(1)$. 
We conclude that $L_3$ and ${\bar L}$ have the same asymptotic distribution
(if any). 

Next  we prove that ${\bar L}$ converges in distribution to (\ref{d*1}). For this purpose we 
 show that  $\E e^{it{\bar L}}\to\E e^{itd_*}$, 
for each $t\in (-\infty,+\infty)$. Denote 
$\Delta(t)=e^{it{\bar L}}-e^{itd_*}$. We shall show below 
that, for any real $t$ and any realized value $Y_1$ 
there exists a positive constant 
$c=c(t,Y_1)$ such that for every $0<\varepsilon<0.5$ 
we have
\begin{equation}\label{epsilon}
 \limsup_{n,m\to+\infty}|\E(\Delta(t)|Y_1)|<c\varepsilon.
\end{equation}
Clearly, (\ref{epsilon}) implies $\E(\Delta(t)| Y_1)=o(1)$.
 This fact together with the simple inequality
 $|\Delta(t)|\le 2$  yields $\E\Delta(t)=o(1)$, 
by Lebesgue's dominated convergence theorem.
Observing that  $\E\Delta(t)=\E e^{it{\bar L}}-\E e^{itd_*}$ we conclude that
$\E e^{it{\bar L}}\to\E e^{itd_*}$.

\bigskip

We fix $0<\varepsilon<0.5$ and prove (\ref{epsilon}). 
Before the proof we introduce  some notation. Denote
\begin{eqnarray}\nonumber
 &&
f_\tau(t)=\E e^{it\tau_1},
\qquad
{\bar f}_{\tau}(t)=\sum_{r\ge 0}e^{itr}{\bar p}_r, 
\qquad
{\bar p}_r
=
{\bar \lambda}^{-1}\sum_{1\le k\le m}\lambda_{k1}{\mathbb I}_{\{{\bar\xi}_k=r\}},
\qquad
{\bar \lambda}=\sum_{k=1}^m\lambda_{k1},
\\
\nonumber
&&
\delta=({\bar f}_\tau(t)-1){\bar \lambda}-(f_{\tau}(t)-1)\lambda_1,
\qquad
f(t)=\E_1e^{itd_*}, 
\qquad
{\bar f}(t)={\bar\E}e^{it{\bar L}}.
\end{eqnarray}
Here ${\bar \E}$ denotes the conditional expectation given $X,Y$ 
and ${\bar\xi}_1,\dots, {\bar\xi}_m$.
Introduce the event
${\cal D}=\{|{\hat a}_1-a_1|<\varepsilon\min\{1, a_1\}\}$ and let ${\overline{\cal D}}$ denote the complement event.  
Furthermore, select the number 
$T>1/\varepsilon$  such that $\PP(\tau_1\ge T)< \varepsilon$. By $c_1, c_2, \dots$ 
we denote positive numbers which do not depend on $n,m$.

We observe that, given $Y_1$, the conditional distribution of $d_*$ is the 
compound Poisson distribution
with the characteristic function $f(t)=e^{\lambda_1(f_{\tau}(t)-1)}$. Similarly,
given $X,Y$ 
and ${\bar\xi}_1,\dots, {\bar\xi}_m$, 
the conditional distribution of ${\bar L}$ is the 
compound Poisson distribution
with the characteristic function ${\bar f}(t)=e^{{\bar \lambda}({\bar f}_{\tau}(t)-1)}$.
In the proof of (\ref{epsilon}) we exploit the convergence
${\bar\lambda}\to\lambda_1$ and ${\bar f}_{\tau}(t)\to f_{\tau}(t)$. In what follows 
we assume that $m,n$ are so 
large that $\beta\le 2m/n\le 4\beta$.

Let us prove  (\ref{epsilon}). 
We write 
\begin{displaymath}
\E_1\Delta(t)=I_{1}+I_{2},
\qquad
I_{1}=\E_1\Delta(t){\mathbb I}_{\cal D},
\qquad
I_2=\E_1\Delta(t){\mathbb I}_{\overline{\cal D}}.         
\end{displaymath}
Here $|I_2|\le 2\PP_1({\overline {\cal D}})=2\PP({\overline {\cal D}})=o(1)$, by the law of large numbers.
Next we estimate $I_{1}$. Combining the identity $\E_1\Delta(t)=\E_1f(t)(e^{\delta}-1)$
with the inequalities $|f(t)|\le 1$ and $|e^{s}-1|\le |s|e^{|s|}$, 
we obtain
\begin{equation}
 |I_1|\le \E_1|\delta|e^{|\delta|}{\mathbb I}_{\cal D}
\le 
c_1\E_1|\delta|{\mathbb I}_{\cal D}.
\end{equation}
Here we estimated $e^{|\delta|}\le e^{8\lambda_1}=:c_1$ using the inequalities
 \begin{displaymath}
 |\delta|\le 2{\bar \lambda}+2\lambda_1,
\qquad
{\bar \lambda}=Y_1(m/n)^{1/2}{\hat a}_1\le 3\lambda_1. 
 \end{displaymath}
We remark that the last  inequality holds for 
$m/n\le 2\beta$ provided that the event ${\cal D}$ occurs.
Finally, we show that $\E_1|\delta|{\mathbb I}_{\cal D}\le (c_2+\lambda_1c_3+\lambda_1c_4)\varepsilon+o(1)$. 
We first write
\begin{displaymath}
\delta
=
({\bar f}_{\tau}(t)-1)({\bar\lambda}-\lambda_1)+({\bar f}_{\tau}(t)-f_{\tau}(t))\lambda_1,
\end{displaymath}
and estimate $|\delta|\le 2|{\bar\lambda}-\lambda_1|+\lambda_1|{\bar f}_{\tau}(t)-f_{\tau}(t)|$.
From the inequalities $|{\hat a}_1-a_1|<\varepsilon$ and $m/n\le 2\beta$ we obtain
\begin{displaymath}
 |{\bar\lambda}-\lambda_1|
\le 
Y_1|{\hat a}_1-a|(m/n)^{1/2}
+Y_1a_1|(m/n)^{1/2}-\beta^{1/2}|
\le 
2Y_1\beta^{1/2}\varepsilon+o(1).
\end{displaymath}
Hence $\E_1|{\bar\lambda}-\lambda_1|{\mathbb I}_{\cal D}\le c_2\varepsilon+o(1)$,
where $c_2=2Y_1\beta^{1/2}$.
We secondly  show that
\begin{displaymath}
 \E_1|{\bar f}_{\tau}(t)-f_{\tau}(t)|{\mathbb I}_{\cal D}\le (c_3+c_4)\varepsilon+o(1).
\end{displaymath}
To this aim we  split 
\begin{displaymath}
 {\bar f}_{\tau}(t)-f_{\tau}(t)=\sum_{r\ge 0}e^{itr}({\bar p}_r-p_r)=R_1-R_2+R_3,
\end{displaymath}
and estimate separately the terms
\begin{displaymath}
 R_1=\sum_{r\ge T}e^{itr}{\bar p}_r,
\qquad
R_2=\sum_{r\ge T}e^{itr}p_r,
\qquad 
R_3=\sum_{0\le r<T}e^{itr}({\bar p}_r-p_r).
\end{displaymath}
Here we denote
$p_r=\PP(\tau_1=r)$. The upper bound for $R_2$ follows by the choice of $T$ 
\begin{displaymath}
 |R_2|\le \sum_{r\ge T}p_r=\PP(\tau_1\ge T)< \varepsilon.
\end{displaymath}
Next,
combining the identity
${\bar p}_r=({\hat a}_1m)^{-1}\sum_{1\le k\le m}X_k{\mathbb I}_{\{{\bar \xi}_k=r\}}$
with the inequalities
\begin{eqnarray}
 |R_1|
\le 
({\hat a}_1m)^{-1}
\sum_{r\ge T}\sum_{1\le k\le m}X_k{\mathbb I}_{\{{\bar \xi}_k=r\}}
=
({\hat a}_1m)^{-1}\sum_{1\le k\le m}X_k{\mathbb I}_{\{{\bar \xi}_k\ge T\}}
\end{eqnarray}
and ${\hat a}_1^{-1}{\mathbb I}_{\cal D}\le 2a_1^{-1}$, we estimate
\begin{equation}\nonumber
 \E_1|R_1|{\mathbb I}_{\cal D}\le 2a_1^{-1}\E_1(X_1{\mathbb I}_{\{{\bar \xi}_1\ge T\}})
\le 
2a_1^{-1}T^{-1}\E_1(X_1{\bar\xi}_1)
=
2a_1^{-1}a_2b_1\beta^{-1/2}\varepsilon.
\end{equation}
Hence $ \E_1|R_1|{\mathbb I}_{\cal D}\le c_4\varepsilon$, where $c_4=2a_1^{-1}a_2b_1\beta^{-1/2}$.

Now we estimate $R_3$. 
We denote $p'_r=({\hat a}_1/a_1){\bar p}_r$
and observe that the inequality $|{\hat a}_1-a_1|\le \varepsilon a_1$ implies
$|{\hat a}_1a_1^{-1}-1|\le \varepsilon$ and
\begin{displaymath}
 |\sum_{0\le r\le T}e^{itr}({\bar p}_r- p'_r)|\le \varepsilon\sum_{0\le r\le T}
{\bar p}_r\le \varepsilon.
\end{displaymath}
In the last inequality we use the fact that the probabilities  $\{{\bar p}_r\}_{r\ge 0}$ 
sum up to $1$. It follows now that
\begin{displaymath}
 |R_3|{\mathbb I}_{\cal D}\le \varepsilon+\sum_{0\le r\le T}|p'_r-p_r|.
\end{displaymath}
Furthermore, observing that 
$\E_1p'_r=a^{-1}\E_1X_k{\mathbb I}_{\{{\bar\xi}_k=r\}}=p_r$, for $1\le k\le m$,
 we obtain 
\begin{displaymath}
 \E_1|p'_r-p_r|^2
= 
m^{-1}\E_1|a^{-1}X_1{\mathbb I}_{\{{\bar\xi}_1=r\}}-p_r|^2\le m^{-1}a_1^{-2}\E X_1^2.
\end{displaymath}
Hence, $\E_1|p'_r-p_r|=O(m^{-1/2})$. We conclude that
\begin{displaymath}
\E_1 |R_3|{\mathbb I}_{\cal D}
\le 
\varepsilon+O(|T|m^{-1/2})=\varepsilon+o(1).
\end{displaymath}

{\it The case (iii)}.  We start with introducing some notation.
Denote $m/n=\beta_n$. Given $\varepsilon\in (0,1)$ 
introduce random variables
% ${\mathbb I}'_{k\varepsilon}={\mathbb I}_{\{X_k(n/m)^{1/2}b_1<\varepsilon\}}.
\begin{displaymath}
\gamma=\sum_{1\le k\le m}{\mathbb I}'_{k}\lambda_{k1}\gamma_{k},
\qquad
\gamma_{k}=X_k\beta_n^{-1/2}b_1{\mathbb I}'_{k},
\qquad
{\mathbb I}'_{k}={\mathbb I}_{\{X_k\beta_n^{-1/2}b_1<\varepsilon\}}.
\end{displaymath}
 Given $X,Y$, 
let ${\tilde {\mathbb I}}_{1},\dots, {\tilde {\mathbb I}}_{m}$ be conditionally independent
Bernoulli random variables with success probabilities 
\begin{displaymath}
{\tilde \PP}({\tilde {\mathbb I}}_{k}=1)=1-{\tilde \PP}({\tilde {\mathbb I}}_{k}=0)=\gamma_{k}.
\end{displaymath}
We assume that, given $X,Y$, 
the sequences 
$\{{\mathbb I}_k\}_{k=1}^m$, 
$\{{\tilde {\mathbb I}}_{k}\}_{k=1}^m$  and 
$\{\xi_{3 k}\}_{k=1}^m$
are conditionally independent. Introduce random variables
\begin{displaymath}
L_4=\sum_{1\le k\le m}{\mathbb I}_k\xi_{3 k},
\qquad
L_5=\sum_{1\le k\le m}{\mathbb I}_k{\mathbb I}'_{k}\xi_{3 k},
\qquad
L_6=\sum_{1\le k\le m}{\mathbb I}_k{\tilde {\mathbb I}}_{k}.
\end{displaymath}
Furthermore, we define the random variable $L_7$ as follows. We first generate $X,Y$. 
Then, given $X,Y$, we
generate a Poisson random variable with the conditional mean value 
$\gamma$. The realized value of the Poisson random variable is denoted $L_7$. Thus, we have
$\PP(L_7=r)=\E e^{-\gamma}\gamma^r/r!$, for $r=0,1,\dots$.

We note that $L$ and $L_3$ have the same asymptotic distribution (if any), 
by (\ref{DeltaT}), (\ref{L-L}). Now we prove 
that $L_3$ converges in distribution to $\Lambda_3$. For this purpose we show  that
for any $\varepsilon\in (0,1)$
\begin{eqnarray}\label{TL1}
&&
d_{TV}(L_3, L_4)=o(1),
\qquad
\E(L_4-L_5)=o(1),
\\
\label{TL2}
&&
d_{TV}(L_5, L_6)\le a_2b_1^2 \varepsilon,
\qquad
d_{TV}(L_6,L_7)=o(1),
\\
\label{TL3}
&&
\E e^{itL_7}-\E e^{it\Lambda_3}=o(1).
\end{eqnarray}

Let us prove (\ref{TL1}), (\ref{TL2}), (\ref{TL3}). The first bound of (\ref{TL1}) 
is obtained in the same way as the first bound of (\ref{DeltaT}). To show the second 
bound  of (\ref{TL1}) we invoke the   inequality
\begin{displaymath}
% \E(L_4-L_5)=
{\tilde \E}(L_4-L_5)
=
\sum_{1\le k\le m}(1-{\mathbb I}_k'){\tilde \E}{\mathbb I}_{k1}{\tilde \E}\xi_{3k}
\le
Y_1b_1m^{-1}\sum_{1\le k\le m}(1-{\mathbb I}_k')X_k^2
\end{displaymath}
and obtain
\begin{displaymath}
\E (L_4-L_5)=\E{\tilde  \E}(L_4-L_5)
\le 
b_1^2\E X_1^2{\mathbb I}_{\{X_1\ge \varepsilon b_1^{-1}\beta_n^{1/2}\}}=o(1). 
\end{displaymath}
We note that the right hand side tends to zero since $\beta_n\to+\infty$.

Let us prove the first inequality of (\ref{TL2}).
Proceeding as in (\ref{dTV4}), (\ref{dTV41}) and using the 
identity ${\tilde {\mathbb I}}_k={\tilde {\mathbb I}}_k{\mathbb I}_k'$  we write
\begin{displaymath}
 {\tilde d}_{TV}(L_5, L_6)
\le 
\sum_{1\le k\le m}
{\mathbb I}_k'
{\tilde \PP}({\mathbb I}_k\not=0)
{\tilde d}_{TV}(\xi_{3k}, {\tilde {\mathbb I}}_k).
\end{displaymath}
Next, we estimate 
${\mathbb I}_k'{\tilde d}_{TV}(\xi_{3k}, {\tilde {\mathbb I}}_k)\le \gamma_k^2$, by
LeCam's inequality (\ref{LeCam}), and invoke the inequality
${\tilde \PP}({\mathbb I}_k\not=0)\le \lambda_{k1}$. 
We obtain
\begin{displaymath}
 {\tilde d}_{TV}(L_5, L_6)
\le 
\sum_{1\le k\le m}
{\mathbb I}_k'\lambda_{k1}\gamma_k^2
\le 
\varepsilon\sum_{1\le k\le m}
{\mathbb I}_k'\lambda_{k1}\gamma_k
\le 
\varepsilon Y_1b_1{\hat a}_2.
\end{displaymath}
Here we estimated $\gamma_k^2\le \varepsilon\gamma_k$. Now the inequalities
$d_{TV}(L_5,L_6)\le \E{\tilde d}_{TV}(L_5,L_6)\le a_2b_1^2\varepsilon$
imply the first relation of (\ref{TL2}).

Let us prove the second relation of (\ref{TL2}).
In view of  
(\ref{fact}) it suffices to show that ${\tilde d}_{TV}(L_6,L_7)=o_P(1)$.
For this purpose we write 
\begin{displaymath}
{\tilde d}_{TV}(L_6,L_7)\le 
{\mathbb I}_{{\cal A}_1}{\tilde d}_{TV}(L_6,L_7)+{\mathbb I}_{{\overline {\cal A}}_1},
\end{displaymath}
 where
${\mathbb I}_{{\overline {\cal A}}_1}=o_P(1)$, see (\ref{d-L22}), and estimate 
using  LeCam's inequality (\ref{LeCam})
\begin{displaymath}
{\mathbb I}_{{\cal A}_1}{\tilde d}_{TV}(L_6,L_7)
\le 
{\mathbb I}_{{\cal A}_1}\sum_{1\le k\le m}{\tilde \PP}^2({\mathbb I}_k{\tilde {\mathbb I}}_k=1){\mathbb I}_k'
\le 
b_1^2Y_1^2m^{-2}{\hat a}_4=o_P(1). 
\end{displaymath}
Here we used the fact that $\E X_1^2<\infty$ implies $m^{-2}{\hat a}_4=o_P(1)$.

Finally, we show (\ref{TL3}). We write
${\tilde \E} e^{itL_7}=e^{\gamma(e^{it}-1)}$ and observe that 
\begin{equation}\label{Y1gamma}
 Y_1b_1a_2-\gamma=o_P(1).
\end{equation}
Furthermore, since for any real $t$ the function $z\to e^{z(e^{it}-1)}$ is bounded and uniformly continuous for $z\ge 0$, 
we conclude that (\ref{Y1gamma}) implies the convergence 
\begin{displaymath}
 \E e^{itL_7}=\E e^{\gamma(e^{it}-1)}\to\E e^{Y_1b_1a_2(e^{it}-1)}=\E e^{it\Lambda_3}.
\end{displaymath}
It remains to prove (\ref{Y1gamma}). 
We write $Y_1b_1a_2-\gamma=Y_1b_1(a_2-{\hat a}_2)+Y_1b_1{\hat a}_2-\gamma$ and 
note that
$a_2-{\hat a}_2=o_P(1)$, by the law of large numbers, and
\begin{displaymath}
 0\le \E (Y_1b_1{\hat a}_2-\gamma)
=
b_1^2\E X_1^2{\mathbb I}_{\{X_1\ge \varepsilon b_1^{-1}\beta_n^{1/2}\}}=o(1).
\end{displaymath}

\end{proof}

\bigskip
 
{\it Acknowledgement}. 
%I am the most grateful to Valentas Kurauskas for his valuable comments and for
%coding and running the simulations.
Research of M. Bloznelis
was supported  by the  Research Council of Lithuania grant MIP-053/2011.  

%Author contributions: M.B. designed research and performed theoretical analysis, V.K.  coded and ran the
%simulations.

%\end{document}

\end{document}